\documentclass{proc-l-hijacked}
\usepackage{amssymb, amsmath, amsthm, hyperref,}

\newtheorem{theorem}{Theorem}[section]
\newtheorem{lemma}[theorem]{Lemma}

\theoremstyle{definition}
\newtheorem{definition}[theorem]{Definition}

\numberwithin{equation}{section}

\DeclareMathOperator{\soc}{soc}
\DeclareMathOperator{\rank}{rank}
\DeclareMathOperator{\tr}{tr}
\DeclareMathOperator{\dett}{det}

\begin{document}
	
	\title[]{Rank in Banach Algebras: A Generalized Cayley-Hamilton Theorem}
	\author{G. Braatvedt, R. Brits and F. Schulz}
	\address{Department of Mathematics, University of Johannesburg, South Africa}
	\email{gabraatvedt@uj.ac.za, rbrits@uj.ac.za,  francoiss@uj.ac.za}

\begin{abstract}
	Let $A$ be a semisimple Banach algebra with non-trivial, and possibly infinite-dimensional socle. Addressing a problem raised in \cite[p.1399]{harher05}, we first define a characteristic polynomial for elements belonging to the socle, and we then show that a Generalized Cayley-Hamilton Theorem holds for the associated polynomial. The key arguments leading to the main result follow from the observation that a purely spectral approach to the theory of the socle carries alongside it an efficient method of dealing with relativistic problems associated with infinite-dimensional socles.
\end{abstract}
	\parindent 0mm
	
	\maketitle

\section{The Characteristic Polynomial}

Let $A$ be a complex, semisimple Banach algebra with identity element $\mathbf 1$ and invertible group $A^{-1}$. For $x\in A$ denote
 $\sigma_A(x):=\{\lambda\in\mathbb C:\lambda\mathbf1-x\notin A^{-1}\}$, and $\sigma_A^\prime(x):=\sigma_A(x)\backslash\{0\}$. If the underlying algebra is clear from the context, then we shall agree to omit the subscript $A$ in the notation $\sigma_A(x)$ and $\sigma_A^\prime(x)$. This convention will also be followed in the forthcoming definitions of rank, trace, determinant, etc.  As in \cite{harher05}, following Aupetit and Mouton in \cite{aupmou96},  we define the \emph{rank} of $a\in A$ by
\begin{equation}\label{rank}
\rank_A(a)=\sup_{x\in A}\# \sigma^\prime(xa)\leq\infty.
\end{equation}
where the symbol $\#K$ denotes the number of distinct elements in a set $K\subseteq \mathbb C$. It can be shown \cite[Corollary 2.9]{aupmou96} that the socle, written $\soc(A)$, of a semisimple Banach algebra $A$ coincides with the collection $\mathcal F:=\{a\in A:\rank(a)<\infty\}$ of finite-rank elements.  With respect to ~\eqref{rank} it is further useful to know that $\sigma^\prime(xa)=\sigma^\prime(ax)$ (Jacobson's Lemma).
If $x\in A$ is such that $\#\sigma^\prime(xa)=\rank(a)$, then we say \emph{$a$ assumes its rank at $x$}. An important fact
in this regard is that, for each $a\in\soc(A)$, the set
\begin{equation}\label{assume}
E_A(a)=\{x\in A:\#\sigma^\prime(xa)=\rank(a)\}
\end{equation}
is dense and open in $A$ \cite[Theorem 2.2]{aupmou96}. If $a\in\soc(A)$ assumes its rank at $\mathbf 1$ then $a$ is said to be a \emph{maximal finite-rank element}. Maximal finite-rank elements are important because they can be ``diagonalized" \cite[Theorem 2.8]{aupmou96}. That is, if $a\in\soc(A)$ satisfies $\rank(a)=\#\sigma^\prime(a)=n$, then $a$ can be expressed as
$$a=\lambda_1p_1+\cdots+\lambda_np_n,$$ where: the $\lambda_i$ are the distinct nonzero spectral values of $a$; and the $p_i$ the corresponding Riesz projections, all of which are minimal (and hence rank one). Furthermore, the collection of maximal finite-rank elements is dense in $\soc(A)$.

For $a\in\soc(A)$, Aupetit and Mouton now use the ``spectral rank" in ~\eqref{rank} to define the \emph{trace} and \emph{determinant} as:
\begin{equation}\label{trace}
\tr_A(a)=\sum\limits_{\lambda\in\sigma(a)}\lambda\, m(\lambda,a)
\end{equation}
\begin{equation}\label{determinant}
\dett_A(a+\mathbf 1)=\prod\limits_{\lambda\in\sigma(a)}(\lambda+1)^{m(\lambda,a)}
\end{equation}
where $m(\lambda,a)$ is the \emph{multiplicity of $a$ at $\lambda$}. A brief description of the notion of multiplicity in the abstract case goes as follows (for particular details one should consult \cite{aupmou96}): Let $a\in\soc(A)$, $\lambda\in\sigma(a)$ and let $V_\lambda$ be an open disk centered at $\lambda$ such that $V_\lambda$ contains no other points of $\sigma(a)$. In \cite[p.119--120]{aupmou96} it is shown  that there exists an open ball, say $U\subset A$, centered at $\mathbf 1$ such that $\#\left[\sigma(xa)\cap V_\lambda\right]$ is constant as $x$ runs through $E(a)\cap U$. This constant integer is the multiplicity of $a$ at $\lambda$. If $\lambda\not=0$ then one can moreover prove that $m(\lambda,a)$ is the rank of the Riesz projection associated to the pair $(\lambda, a)$. If $a$ is a maximal finite-rank element then $m(\lambda,a)=1$ \cite[p.120]{aupmou96}.

In the operator case, $A=\mathcal L(X)$, where $X$ is a Banach space, the formulas in ~\eqref{rank},  ~\eqref{trace}, and ~\eqref{determinant} can be shown to coincide with the respective classical operator definitions. The Aupetit-Mouton approach is not merely an alternative to the long established theory of rank, trace and determinant for $A=\mathcal L(X)$. It extends the classical theory because it simultaneously takes care of the matter in subalgebras of $\mathcal L(X)$ as well; the notions of rank, trace, and determinant are clearly relative concepts.

  To generalize the Cayley-Hamilton Theorem for matrices \cite[Theorem 3.3.2]{aup91} to the socle of an arbitrary Banach algebra we need a suitable candidate for the characteristic polynomial associated with an element $a\in\soc(A)$.   

\begin{definition}\textnormal{(Generalized Characteristic Polynomial)}\label{d.1}
Let $a \in \soc(A)$. The \emph{generalized characteristic polynomial} of $a$ is defined to be the complex polynomial
\begin{equation}\label{jr}
p_a(\lambda)=\prod_{\alpha \in \sigma_A(a)} \left(\alpha-\lambda\right)^{m\left(\alpha, a\right)}.
\end{equation}
where $m(\alpha,a)$ is the \emph{spectral} multiplicity of $a$ at $\alpha$ (described in the preceding paragraph).
\end{definition}
If $a\in\soc(A)$, $p_a(\lambda)$ as in (\ref{jr}),  and if $x$ belongs to a Banach algebra $B$ with identity $e$ then we also define $$p_{a,e}(x)=\prod_{\alpha \in \sigma_A(a)} \left(\alpha e-x\right)^{m\left(\alpha, a\right)}$$
with the understanding that if $B=A$ and $e=\mathbf 1$ we simply write $p_a(x)$. 
Definition~\ref{d.1} calls for some comments: To start with, the product defined in Definition~\ref{d.1} has a finite number of factors since $a\in\soc(A)$ (which implies that the spectrum of $a$, and the associated multiplicities are finite). Thus the polynomial $p_a(\lambda)$ exists.  Moreover, for any fixed $\lambda_0 \in \mathbb{C}$, using a similar argument as in the proof of
 \cite[Theorem 3.3]{aupmou96}, it can be shown that $a \mapsto p_a(\lambda_0)$ is continuous on $\mathcal{F}_{k}:=\{a\in\soc(A):\rank(a)\leq k\}$ for every nonnegative integer $k$.
It is important to realize that, in Definition~\ref{d.1},  Aupetit and Mouton's notion of multiplicity is independent of the particular structure and dimension of the socle, and that it should therefore not be compared to the classical algebraic or geometric multiplicities of eigenvalues in the case where $A=M_n(\mathbb C)$.
Specifically, if $a \in M_{n} \left(\mathbb{C}\right)$, then it is not necessarily the case that $p_a(\lambda)$ is equal to the characteristic polynomial as defined in the classical sense. This is immediately obvious if one considers $a=0$. To give a non-trivial example, if $a\in M_3(\mathbb C)$ is
$$a=\left(\begin{array}{ccc} 1 & 0 & 0 \\ 0 & 0 & 0 \\ 0 & 0 & 0 \end{array}\right),$$
then $p_a(\lambda) = -\lambda\left(1-\lambda\right)$, whereas the classical characteristic polynomial of $a$ is given by $\dett(a-\lambda\mathbf 1) =(-\lambda)^{2}\left(1-\lambda\right)$. The explanation for this follows from observing that the Aupetit-Mouton definition of multiplicity (\cite[p.120]{aupmou96}) of $0\in\sigma(a)$ does not necessarily coincide with the algebraic multiplicity associated with the $0$ spectral value of a matrix (in the case of singular matrices). However, if $a \in M_{n} \left(\mathbb{C}\right)$ is an invertible maximal finite-rank element, then $0 \notin \sigma (a)$ and $\#\sigma^\prime(a)= n$, so the algebraic multiplicity of each spectral value of $a$ is $1$. Hence, in this case, it follows that $p_a(\lambda)$ does in fact coincide with the classical characteristic polynomial of $a$. Despite the aforementioned discrepancy one observes that, for $a\in\soc(A)$, the characteristic polynomial in Definition~\ref{d.1} encodes all information pertaining to the spectral values of $a$, as well as their multiplicities, but in the context of the generalized definitions given in \cite{aupmou96}. So it is reasonable to conjecture that $p_a(a)=0$.

\section{The Cayley-Hamilton Theorem}

To avoid any chance of confusion, relative identity elements belonging to the same algebra will be clearly indicated. Further, the notation $\dett_C(\cdot)$ which appears in this section refers exclusively to the classical determinant where $C=M_n(\mathbb C)$ for some $n$; so the determinant in (\ref{determinant}) will not be used. In order to prove the main result, Theorem~\ref{f.5}, we need a little preparation: Lemma~\ref{f.0} is well-known, and the first part appears in \cite[Chapter 3, Exercise 6]{aup91}; we have been unable to find a suitable reference for \eqref{eqf.2}, but the proof is not hard:
\begin{lemma}\label{f.0}Let $p$ be a projection of a complex, semisimple, and unital Banach algebra $A$. Then $pAp$ is a closed semisimple subalgebra of $A$ with identity $p$ and
\begin{equation}\label{eqf.2}
\sigma^\prime_{pAp} \left(pxp\right) = \sigma^\prime_{A} \left(pxp\right)
\end{equation}
for each $x \in A$.
\end{lemma}

\begin{proof}
If $$(pxp-\lambda p)pyp=pyp(pxp-\lambda p)=p,$$ then
$$(pxp-\lambda \mathbf 1)\left(\frac{1}{\lambda}(p-\mathbf 1)+pyp\right)=\mathbf 1=\left(\frac{1}{\lambda}(p-\mathbf 1)+pyp\right)(pxp-\lambda \mathbf 1).$$
Conversely, if $$(pxp-\lambda \mathbf 1)y=y(pxp-\lambda \mathbf 1)=\mathbf 1,$$ then
$$(pxp-\lambda p)pyp=pyp(pxp-\lambda p)=p.$$
\end{proof}

\begin{lemma}\label{f.1}
Let $p$ be a finite-rank projection of $A$. Then
$$\mathrm{rank}_{pAp} \,(pxp) = \mathrm{rank}_{A}\,(pxp)$$
for each $x \in A$.
\end{lemma}

\begin{proof}
Let $x \in A$ be arbitrary. It readily follows from Lemma~\ref{f.0}, and Jacobson's Lemma, that
\begin{eqnarray*}
\mathrm{rank}_{pAp}\,(pxp) &=& \sup_{y\in A} \#\sigma^\prime_{pAp} \left((pyp)(pxp)\right) \\
&=& \sup_{y\in A} \#\sigma^\prime_{A} \left((pyp)(pxp)\right) \\
&=& \sup_{y\in A} \#\sigma^\prime_{A} \left(y(pxp)\right) \\
&=& \mathrm{rank}_{A}\,(pxp),
\end{eqnarray*}
as desired.
\end{proof}

\begin{lemma}\label{f.2}
Let $A_{j} = M_{n_{j}} \left(\mathbb{C}\right)$ for each $j \in \left\{1, \ldots, k \right\}$ and let $A = A_{1} \oplus \cdots \oplus A_{k}$. Suppose that $a = \left(a_{1}, \ldots, a_{k}\right)$ is a maximal finite-rank element of $A$. Then
$\sigma^\prime_{A_{i}} \left(a_{i}\right)\cap \sigma^\prime_{A_{j}} \left(a_{j}\right)=\emptyset$ for $i \neq j$.
\end{lemma}

\begin{proof}
Assume, to the contrary, that
$$\sigma^\prime_{A_{i}} \left(a_{i}\right)\cap \sigma^\prime_{A_{j}} \left(a_{j}\right)\neq\emptyset$$
for some $i, j \in \left\{1, \ldots, k \right\}$ with $i \neq j$. Let
$$q_{0} = \min\left\{\left|\alpha - \beta\right| : \alpha, \beta \in \sigma_{A} (a) \cup \left\{0 \right\}, \alpha \neq \beta  \right\}$$
and let $q = q_{0}/2$. Let $\left(\alpha_{n}\right) \subseteq \left(0, 1\right)$ be any sequence such that $\alpha_{n} \rightarrow 0$ as $n \rightarrow \infty$. Then $\left(1 - \alpha_{M}\right) \beta \in B\left(\beta, q\right)$ for each $\beta \in \sigma_{A} (a) \cup \left\{0 \right\}$ if $M$ is taken sufficiently large. However, since
$$\sigma_{A} \left(\left(w_{1}, \ldots, w_{k}\right)\right) = \bigcup_{j=1}^{k} \sigma_{A_{j}} \left(w_{j}\right)$$
for each $\left(w_{1}, \ldots, w_{k}\right) \in A$, it readily follows that
\begin{eqnarray*}
&& \#\sigma^\prime_{A} \left(\left(\mathbf{1}, \ldots, \mathbf{1}, \left(1-\alpha_{M}\right)\mathbf{1}, \mathbf{1}, \ldots, \mathbf{1}\right)a\right)\\
& = & \#\sigma^\prime_{A} \left(\left(a_{1}, \ldots, a_{j-1}, \left(1-\alpha_{M}\right)a_{j}, a_{j+1}, \ldots, a_{k}\right)\right)\\
& > & \#\sigma^\prime_{A} (a).
\end{eqnarray*}
But then we obtain a contradiction with the fact that $a$ is a maximal finite-rank element of $A$. So the lemma is true.
\end{proof}

\begin{lemma}\label{f.3}
Let $A_{j} = M_{n_{j}} \left(\mathbb{C}\right)$ for each $j \in \left\{1, \ldots, k \right\}$ and let $A = A_{1} \oplus \cdots \oplus A_{k}$. Suppose that $a = \left(a_{1}, \ldots, a_{k}\right)$ is a maximal finite-rank element of $A$. Then $a_{j}$ is a maximal finite-rank element of $A_{j}$ for each $j \in \left\{1, \ldots, k \right\}$.
\end{lemma}

\begin{proof}
Assume, to the contrary, that $\#\sigma^\prime_{A_{j}} \left(a_{j}\right) < \mathrm{rank}_{A_{j}}\left(a_{j}\right)$ for some $j \in \left\{1, \ldots, k \right\}$. Let $y \in E\left(a_{j}\right)$. If $\sigma^\prime_{A_{j}}\left(ya_{j}\right)\cap\sigma^\prime_{A_{i}} \left(a_{i}\right)\neq\emptyset$ for some $i \neq j$, then we may apply the argument in the proof above to obtain a real number $\alpha > 0$ such that $\sigma^\prime_{A_{j}}\left(\alpha ya_{j}\right)\cap\sigma^\prime_{A_{i}} \left(a_{i}\right)=\emptyset$ for all $i \neq j$. But then
$$\#\sigma^\prime_{A} \left(\left(a_{1}, \ldots, a_{j-1}, \alpha ya_{j}, a_{j+1}, \ldots, a_{k}\right)\right) > \#\sigma^\prime_{A}(a),$$
so we obtain a contradiction with the maximality assumption on $a$. The result now follows.
\end{proof}

\begin{lemma}\label{f.4}
Let $A_{j} = M_{n_{j}} \left(\mathbb{C}\right)$ for each $j \in \left\{1, \ldots, k \right\}$ and let $A = A_{1} \oplus \cdots \oplus A_{k}$. Suppose that $a = \left(a_{1}, \ldots, a_{k}\right)$ is an invertible maximal finite-rank element of $A$. Then
\begin{equation}
p_a(\lambda) = \prod_{j=1}^{k}\dett_{A_{j}}\left(a_{j}- \lambda \mathbf{1_j}\right)
\label{eqf.1}
\end{equation}
for all $\lambda \in \mathbb{C}$.
\end{lemma}

\begin{proof}
By Lemma \ref{f.2}, Lemma \ref{f.3} and the fact that $a \in A^{-1}$, it readily follows that $\beta \in \sigma_{A}(a) \cap \sigma_{A_{j}} \left(a_{j}\right)$ implies that $m \left(\beta, a\right) = m\left(\beta, a_{j}\right) = 1$. Consequently, (\ref{eqf.1}) holds true.
\end{proof}

\begin{theorem}\label{f.5}
\textnormal{(Generalized Cayley-Hamilton Theorem)} Let $a \in \soc(A)$ and let $p_a \left(\lambda\right)$ be its generalized characteristic polynomial. Then $p_a(a) = 0$.
\end{theorem}

\begin{proof}
If $a = 0$, the result trivially holds true. So assume that $a \neq 0$. By hypothesis and \cite[Corollary 2.9]{aupmou96}, $a$ has finite-rank, say $\mathrm{rank}\,(a) = n \geq 1$. Suppose first that $a$ is a maximal finite-rank element of $A$ and that $a \notin A^{-1}$. By Theorem \cite[Theorem 2.8]{aupmou96} there exist orthogonal minimal projections $p_{1}, \ldots, p_{n} \in A$ such that $a= \lambda_{1}p_{1} + \cdots \lambda_{n}p_{n}$, where $\lambda_{1}, \ldots, \lambda_{n}$ are the distinct nonzero spectral values of $a$. By the orthogonality and minimality of the $p_{i}$ it readily follows that $e = p_{1} + \cdots +p_{n}$ is a finite-rank projection of $A$. Thus, by Lemma~\ref{f.0}, and \cite[Lemma 4.2]{Puhl}, it follows that $B = eAe$ is a finite-dimensional semisimple closed subalgebra of $A$ with identity $e$, and moreover that
$$\sigma^\prime_{A} (a)=\sigma^\prime_{B} (a).$$
 Observe now that $a\in B^{-1}$ and that, by Lemma~\ref{f.1}, $a$ is a maximal finite-rank element of $B$. In particular, this implies that the multiplicity of each nonzero spectral value of $a$ is $1$, regardless of whether $a$ is viewed as an element of $A$ or $B$ (notice further that $0\in\sigma_A(a)$ has multiplicity one, whereas $0\notin\sigma_B(a)$).  Also, by the Wedderburn-Artin Theorem \cite[Theorem 2.1.2]{aup91} it follows that $B$ is isomorphic as an algebra to
$$C = M_{n_{1}} \left(\mathbb{C}\right) \oplus \cdots \oplus M_{n_{k}} \left(\mathbb{C}\right).$$
Let $\psi$ be the algebra isomorphism from $B$ onto $C$, let $\psi (a) = \left(a_{1}, \ldots, a_{k}\right)$, and let $\psi (e)=\mathbf e$ be the identity of $C$. For each $j \in \left\{1, \ldots, k \right\}$, let $C_{j} = M_{n_{j}} \left(\mathbb{C}\right)$. Using Lemma \ref{f.4} it follows that
\begin{align*}
p_a(\lambda)&=-\lambda\prod_{\alpha \in \sigma_B(a)} \left(\alpha-\lambda\right)^{m\left(\alpha, a\right)}=-\lambda p_{\psi(a)}(\lambda) = -\lambda \prod_{j=1}^{k}\dett_{C_{j}}\left(a_{j}- \lambda \mathbf{1_j}\right)
\end{align*}
Furthermore, since each $\dett_{C_{j}}\left(a_{j}- \lambda \mathbf{1_j}\right)$ defines a polynomial on $\mathbb{C}$, it follows that $\lambda \mapsto \dett_{C_{j}}\left(a_{j}- \lambda \mathbf{1_j}\right)$ is an entire function for each $j \in \left\{1, \ldots, k \right\}$. Let $\Gamma$ be the union of $n+1$ disjoint circles with centers respectively at $\lambda_{1}, \ldots, \lambda_{n}$ and $0$. Now for each $j \in \left\{1, \ldots, k \right\}$ and $\lambda \in \Gamma$ we have
$$\left(a_{j} - \lambda\mathbf{1_j}\right)^{-1} = \frac{1}{\dett_{C_{j}}\left(a_{j}- \lambda \mathbf{1_j}\right)} b_{j} \left(\lambda\right),$$
where $b_{j} \left(\lambda\right)$ is a $n_{j} \times n_{j}$ matrix depending analytically on $\lambda$ since its $\left(k, l\right)$-entry is the $\left(l, k\right)$-cofactor of $a_{j}- \lambda \mathbf{1_j}$, and so it is a polynomial in $\lambda$ of degree less than or equal to $n_{j}-1$. For $j \in \left\{1, \ldots, k \right\}$, let $e_{j}\left(\lambda\right)$ denote the element of $C$ which takes the value $b_{j} \left(\lambda\right)$ at the $j$th coordinate and the value $0$ at all other coordinates. Then, since
$$\left(\lambda \mathbf{e} - \psi (a)\right)^{-1} = - \left( \left(a_{1}-\lambda\mathbf{1_1}\right)^{-1}, \ldots, \left(a_{k}-\lambda\mathbf{1_k}\right)^{-1}\right),$$
we obtain that
$$p_{a,\mathbf e}\left(\psi (a)\right) = \frac{1}{2 \pi i} \sum_{j=1}^{k}\left[\int_{\Gamma} \lambda \left(\prod_{i \neq j}\dett_{C_{i}}\left(a_{i}- \lambda \mathbf{1_i}\right)\right) e_{j} \left(\lambda\right)\,d\lambda\right].$$
But for each $j \in \left\{1, \ldots, k \right\}$, using the standard basis for $C_{j}$ and Cauchy's Theorem, we have
$$ \int_{\Gamma} \lambda \left(\prod_{i \neq j}\dett_{C_{i}}\left(a_{i}- \lambda \mathbf{1_i}\right)\right) e_{j} \left(\lambda\right)\,d\lambda = 0.$$
Thus, $p_{a,\mathbf e}\left(\psi (a)\right) = 0$. Consequently $p_{a,e}(a) = 0$ in $B$, and since the expression $p_a(\lambda)$ does not contain a constant term we also have $p_{a}(a) = 0$ in $A$. If $a$ is a maximal finite-rank element of $A$ and $a \in A^{-1}$, then, in particular, $\soc(A) = A$ implying that $A$ is finite-dimensional. Thus, we may apply the Wedderburn-Artin Theorem directly to $A$, and use a similar argument as above to conclude that $p_a(a)=0$ in $A$. Here $p_a\left(\lambda\right)$ does have a constant term. However, the identity element used in $p_a(a)$ is that of $A$ since we did not pass to a subalgebra of $A$. So the result is true if $a$ is a maximal finite-rank element of $A$. Suppose now that $a$ is not a maximal finite-rank element of $A$. Let $\Gamma_{0} = \partial B\left(0, r\right)$, where $r > 0$ is chosen sufficiently large so that $\sigma_{A} (a) \subseteq B\left(0, r\right)$. Using the upper semicontinuity of the spectrum and \cite[Theorem 2.2]{aupmou96}, we can find a sequence $\left(x_{m}\right) \subseteq E(a)$ such that $x_{m} \rightarrow \mathbf{1}$ as $m \rightarrow \infty$ and $\sigma_{A} \left(x_{m}a\right)$ is contained in the interior of $\Gamma_{0}$ for each integer $m \geq 1$. For each integer $m \geq 1$, denote by $p_{m} \left(\lambda\right) :=p_{x_ma} \left(\lambda\right)$ the characteristic polynomial of $x_{m}a$. Applying the preceding argument to the maximal finite-rank element $x_{m}a$, we conclude that $p_{m} \left(x_{m}a\right) = 0$ for each integer $m \geq 1$. Now, by the continuity of the determinant on $\mathcal{F}_{n}$ it follows that $\left(p_{m}\right)$ converges to $p_a$ pointwise on $\mathbb{C}$. Moreover, by compactness of $\Gamma_{0}$, and continuity of the resolvent on $A^{-1}$, we may infer the existence of two positive real numbers $K_{1}$ and $K_{2}$ such that
$$\left\| \left(\lambda\mathbf{1}-x_{m}a\right)^{-1}\right\| \leq K_{1}$$
and
$$\left\| \left(\lambda\mathbf{1}-x_{m}a\right)^{-1} - \left(\lambda\mathbf{1}-a\right)^{-1} \right\| \leq K_{2}$$
for each $\lambda \in \Gamma_{0}$ and each integer $m \geq 1$. In addition, since $p_a$ is continuous and $\Gamma_{0}$ is compact, it follows that $\left|p_a(\lambda)\right|$ is bounded on $\Gamma_{0}$, say $\left|p_a(\lambda)\right| \leq K_{3}$ for each $\lambda \in \Gamma_{0}$. Also, for each $\lambda \in \Gamma_{0}$ and integer $m \geq 1$ we have
\begin{eqnarray*}
&&\left|p_{m} (\lambda) - p_a(\lambda)\right|  \leq  \left|p_{m} (\lambda)\right| + \left|p_a (\lambda)\right| \leq \left|p_{m} (\lambda)\right| + K_{3} \\
& = & \left|\dett\left(x_{m}a-\lambda \mathbf{1}\right) \right| + K_{3} = \left( \prod_{\alpha \in \sigma_{A} \left(x_{m}a\right)} \left|\alpha - \lambda\right| \right) + K_{3} \\
& \leq & \left( \prod_{\alpha \in \sigma_{A} \left(x_{m}a\right)} \left(\left|\alpha\right| + \left|\lambda\right|\right) \right) + K_{3}  \leq  \left(\rho \left(x_{m}a\right) + r\right)^{\mathrm{rank}\,(a)+1} + K_{3} \\
& \leq & \left(r + r\right)^{n+1} + K_{3} = 2^{n+1}r^{n+1} + K_{3} = K_{4}.
\end{eqnarray*}
Consequently, writing $q_{m}(\lambda) = p_{m} (\lambda) - p_a(\lambda)$ and $F \left(\lambda, m\right) = \left(\lambda\mathbf{1}-x_{m}a\right)^{-1} - \left(\lambda\mathbf{1}-a\right)^{-1}$, we obtain
\begin{eqnarray*}
&&\left\|p_{m} \left(x_{m}a\right) - p_a(a)\right\|   =  \left\|p_{m}\left(x_{m}a\right) - p_a\left(x_{m}a\right) + p_a\left(x_{m}a\right) - p_a(a)\right\| \\
& = & \frac{1}{2 \pi}\left\|\int_{\Gamma_{0}} q_{m}(\lambda)\left(\lambda\mathbf{1}-x_{m}a\right)^{-1}\,d\lambda + \int_{\Gamma_{0}} p_a(\lambda)\cdot F \left(\lambda, m\right)\,d\lambda \right\| \\
& \leq & \frac{1}{2 \pi}\left[\int_{\Gamma_{0}} \left|q_{m}(\lambda)\right|\cdot\left\|\left(\lambda\mathbf{1}-x_{m}a\right)^{-1}\right\|\,d\left|\lambda\right| + \int_{\Gamma_{0}} \left|p_a(\lambda)\right|\cdot\left\|F \left(\lambda, m\right)\right\|\,d\left|\lambda\right| \right] \\
& \leq & \frac{1}{2 \pi}\left[\int_{\Gamma_{0}} \left|q_{m}(\lambda)\right|\cdot K_{1}\,d\left|\lambda\right| + \int_{\Gamma_{0}} K_{3}\cdot\left\|F \left(\lambda, m\right)\right\|\,d\left|\lambda\right| \right].
\end{eqnarray*}
 But, $\left|q_{m} (\lambda)\right| \cdot K_{1} \leq K_{4} \cdot K_{1}$ and $K_{3}\cdot\left\|F \left(\lambda, m\right)\right\| \leq K_{3} \cdot K_{2}$ for each $\lambda \in \Gamma_{0}$ and integer $m \geq 1$, so by the Dominated Convergence Theorem it follows that $p_{m} \left(x_{m}a\right) \rightarrow p_a(a)$ as $m \rightarrow \infty$. However, $p_{m} \left(x_{m}a\right) = 0$ for each integer $m \geq 1$ whence $p_a(a) = 0.$
\end{proof}

\section{Concluding remarks}
 In view of (\ref{determinant}) it is tempting to define, for each $a\in\soc(A)$,
 \begin{equation}\label{jw}
  \dett(a-\lambda\mathbf 1):= \prod_{\alpha \in \sigma (a)} \left(\alpha-\lambda\right)^{m\left(\alpha, a\right)}
 \end{equation}
 where $m(\alpha,a)$ is the \emph{spectral} multiplicity of $a$ at $\alpha$. By definition, we would then have $p_a(\lambda)=\dett(a-\lambda\mathbf 1)$ as in the matrix case. The main reasons which compelled the Authors not to formulate Definition~\ref{d.1} in terms of a determinant are the following:
 \begin{itemize}
 	\item[(i)]{  With the formulation of Definition~\ref{d.1} we were, to some extent, influenced by Sheldon Axler's (somewhat controversial) paper \cite{Axler}. In particular, for $A=M_n(\mathbb C)$, Axler defines the characteristic polynomial, and proves the Cayley-Hamilton Theorem, without the notion of a determinant. Definition~\ref{d.1} is precisely Axler's definition but with the multiplicities replaced by Aupetit and Mouton's ``spectral multiplicities". The use of matrix determinants in the proofs of the results in Section 2 was merely a matter of convenience, since in each instance where a determinant appears, the particular expression equals the characteristic polynomial of some matrix in the sense of Axler. It therefore seemed plausible to obtain a ``determinant free" characteristic polynomial and a subsequent Cayley-Hamilton Theorem. } 
 	\item[(ii)]{ To avoid possible confusion; the example given in Section 1 of the current paper clearly illustrates the point.}
 	\item[(iii)]{ Related to (ii) above, one would expect a formula which is called, and denoted, a determinant to have the basic properties of the classical determinant; with the definition~(\ref{jw}), take $A=\mathbb C^3,$ $a=(1,1,0)\in A$ and observe that 
 	\begin{align*}	
 		 \dett(a-2\mathbf1)&=\dett\left(\left({a}/{2}-\mathbf1\right)2\mathbf 1\right)\\&\not=\dett\left({a}/{2}-\mathbf1\right)\det(2\mathbf1).
 	\end{align*}
 		 So, with (\ref{jw}), the determinant might not be multiplicative. This problem does not surface with Aupetit and Mouton's formulation of the determinant in (\ref{determinant}) because, obviously, $1^n=1$ for all $n\in\mathbb N$. If $a,b\in\soc(A)$, then with (\ref{determinant}) we do have the multiplicative property \cite[Theorem 3.3]{aupmou96} $$\dett((a+\mathbf 1)(b+\mathbf 1))=\dett(a+\mathbf 1)\dett(b+\mathbf 1)$$ as well as a Generalized Sylvester's Theorem \cite[Theorem 2.4]{bbsdet} 
 		 $$\dett(ab+\mathbf 1)= \dett(ba+\mathbf 1).$$      }
\end{itemize}

 \bibliographystyle{amsplain}

\end{document}